\documentclass[twoside,reqno]{amsart}
\usepackage{amssymb,amsmath,amsfonts,amsthm}
\usepackage{enumerate}
\usepackage{array}
\usepackage[colorlinks=true]{hyperref}
\hypersetup{urlcolor=blue, citecolor=red}

\newtheorem{theorem}{Theorem}

\newtheorem{remark}[theorem]{Remark}

\newtheorem{proposition}[theorem]{Proposition}
\newtheorem{lemma}[theorem]{Lemma}

\newcommand{\N}{{\mathbb N}}

\newcommand{\HH}{\mathbb H}
\newcommand{\R}{{\mathbb R}}

\newcommand{\T}{{\mathbb T}}

\newcommand{\p}{\partial}

\def\Ec{\mathcal E }

\def\ds{\displaystyle}
\def\eps{\varepsilon}
\def\vphi{\varphi}

\def\d{\delta}

\def\t{\tau}

\usepackage{mathtools}

\title[Data dependence of approximate current-vortex sheets]{Data dependence of approximate current-vortex sheets\\ near the onset of instability}
\author[A.~Morando] {Alessandro Morando}
\address{DICATAM, Sezione di Matematica, \newline \indent
Universit\`a di Brescia,\newline \indent Via Valotti, 9, 25133 BRESCIA, Italy}
\email{alessandro.morando@unibs.it, paolo.secchi@unibs.it, paola.trebeschi@unibs.it}

\author[P.~Secchi]{Paolo Secchi}

\author[P.~Trebeschi]{Paola Trebeschi}


\thanks{The authors are supported by the national research project PRIN 2012 \lq\lq Nonlinear Hyperbolic Partial Differential Equations, Dispersive and Transport Equations: theoretical and applicative aspects\rq\rq.}

\date{\today}

\begin{document}

\begin{abstract}
The paper is concerned with the free boundary problem for 2D current-vortex sheets in ideal incompressible
magneto-hydrodynamics near the transition point between the linearized stability and instability.
In order to study the dynamics of
the discontinuity near the onset of the instability, Hunter and Thoo \cite{hunter-thoo} have introduced an asymptotic quadratically nonlinear integro-differential equation for the amplitude of small perturbations of the planar discontinuity.

The local-in-time existence of smooth solutions to the Cauchy problem for the amplitude equation was shown in \cite{M-S-T:ONDE1, M-S-T:ONDE2}. In the present paper we prove the continuous dependence in strong norm of solutions on the initial data. This completes the proof of the well-posedness of the problem in the classical sense of Hadamard.

\end{abstract}

\maketitle

\section{Introduction}
\label{sect1}


In the present paper we consider the following equation
\begin{equation}\label{equ1}
\varphi_{tt}-\mu\varphi_{xx}=\left(\frac12\mathbb H[\phi^2]_{xx}+\phi\varphi_{xx}\right)_{\!\!x}\,,\qquad\phi=\mathbb H[\varphi]\,,
\end{equation}
where the unknown is the scalar function $\varphi=\varphi(t,x)$, where $t$ denotes the time, $x\in\R$ is the space variable and $\mathbb H$ denotes the Hilbert transform with respect to $x$, and $\mu$ is a constant. Hunter and Thoo \cite{hunter-thoo} have derived this asymptotic equation in order to study the dynamics of 2D current-vortex sheets in ideal incompressible
magneto-hydrodynamics, near the transition point between the linearized stability and instability.

Equation \eqref{equ1} is an integro-differential evolution equation in one space dimension, with quadratic nonlinearity. 
This is a nonlocal equation of order two: in fact, it may also be written as
\begin{equation}
\begin{array}{ll}\label{equ1bis}
\vphi_{tt}-\mu\vphi_{xx} = \left(  [ \HH;\phi
]\partial_x  \phi
_{x} + \HH[\phi
^2_x]\right)_x \,,
\end{array}
\end{equation}
where $[ \HH;\phi
]\partial_x$ is a pseudo-differential operator of order zero. This alternative form \eqref{equ1bis} shows that \eqref{equ1} is an equation of second order, due to a cancelation of the third order spatial derivatives appearing in \eqref{equ1}.

Equation \eqref{equ1} also admits the alternative spatial form
\begin{equation}\label{equ1ter}
\begin{split}
\varphi_{tt}-\left(\mu-2\phi_x\right)\varphi_{xx}+\mathcal Q\left[\varphi\right]=0\,,
\end{split}
\end{equation}
where 
\begin{equation}\label{termine_nl}
\mathcal Q\left[\varphi\right]:=-3\left[\mathbb H\,;\,\phi_x\right]\phi_{xx}-\left[\mathbb H\,;\,\phi\right]\phi_{xxx}\,.
\end{equation}
The alternative form \eqref{equ1ter} puts in evidence the difference $\mu-2\phi_x$ which has a meaningful role. In fact it can be shown that the linearized operator about a given basic state is elliptic and \eqref{equ1} is locally linearly ill-posed in points where 
\[
\mu-2\phi_x<0.
\]
On the contrary, in points where
\begin{equation}
\begin{array}{ll}\label{extstab}

\mu-2\phi
_{x}>0

\end{array}
\end{equation}
the linearized operator is hyperbolic and \eqref{equ1} is locally linearly well-posed, see \cite{hunter-thoo}. In this case we can think of  \eqref{equ1ter} as a nonlinear perturbation of the wave equation.

The local-in-time existence of smooth solutions to the Cauchy problem for \eqref{equ1}, under the above stability condition \eqref{extstab}, was shown in \cite{M-S-T:ONDE1, M-S-T:ONDE2}. In the present paper we prove the continuous dependence in strong norm of solutions on the initial data. This completes the proof of the well-posedness of \eqref{equ1} in the classical sense of Hadamard, after existence and uniqueness. 

As written in Kato's paper \cite{kato75}, this part may be the most difficult one, when dealing with hyperbolic problems. Our method is somehow inspired by Beir{\~a}o da Veiga's perturbation theory for the compressible Euler equations \cite{beirao92,beirao93}, and its application to the problem of convergence in strong norm of the incompressible limit, see \cite{beirao1994,secchisingular}.
Instead of directly estimating the difference between the given solution and the solutions to the problems with approximating initial data, the main idea is to use a triangularization with the more regular solution to a suitably chosen close enough problem.



Let us consider the initial value problem for the equation \eqref{equ1} supplemented by the initial condition
\begin{equation}\label{id}
\varphi_{\vert\,t=0}=\varphi^{(0)}\,,\qquad \partial_t\varphi_{\vert\,t=0}=\varphi^{(1)}\,,
\end{equation}
for sufficiently smooth initial data $\varphi^{(0)},\,\varphi^{(1)}$ satisfying the stability condition
\begin{equation*}\label{stability_nl}
\mu-2\phi^{(0)}_{x}>0\,,\qquad\phi^{(0)}:=\mathbb H[\varphi^{(0)}]\,.
\end{equation*}

For the sake of convenience, in the paper the unknown $\varphi=\varphi(t,x)$ is a scalar function of  the time $t\in\mathbb R^+$  and the space variable $x$, ranging on the one-dimensional torus $\mathbb T$ (that is $\varphi$ is periodic in $x$). For all notation we refer to the following Section \ref{prbt}.

In \cite{M-S-T:ONDE2} we prove the following existence theorem.
\begin{theorem}\label{th_esistenza}
Let $s\ge 3$ be a real number. Assume $\varphi^{(0)}\in H^{s}(\mathbb T)$, $\varphi^{(1)}\in H^{s-1}(\mathbb T)$, and let $\varphi^{(0)}$, $\varphi^{(1)}$ have zero spatial mean. 
Given $0<\delta<\mu$, there exist $0<R\le1$, and constants $C_1>0$, $C_2>0$ such that, if
\begin{equation}\label{ip_dati_iniziali}
\Vert\varphi^{(0)}_x\Vert^2_{H^2(\T)}+\Vert\varphi^{(1)}\Vert^2_{H^2(\T)}<R^2\,,
\end{equation}
there exists a unique solution $\varphi\in C(I_0; H^{s}(\mathbb T))\cap C^1(I_0; H^{s-1}(\mathbb T))$ of the Cauchy problem \eqref{equ1}, \eqref{id} defined on the time interval $I_0=[0,T_0)$, where
\begin{equation}
\begin{array}{ll}\label{tmax}
\ds T_0=C_1\left( \Vert\varphi^{(0)}_x\Vert^2_{H^2(\T)}+\Vert\varphi^{(1)}\Vert^2_{H^2(\T)} \right)^{-1/2} .
\end{array}
\end{equation}
The solution $\varphi$ has zero spatial mean and satisfies, for all $t\in I_0$,
\begin{equation}\label{stab1}
\mu-2\phi_x\ge\delta\,,
\end{equation}
 \begin{equation}\label{stima_Hs}
    \Vert\varphi(t)\Vert^2_{H^s(\T)}+\Vert\varphi_t(t)\Vert^2_{H^{s-1}(\T)}\le {C_2}  \left(\Vert\varphi^{(0)}\Vert^2_{H^{s}(\T)}+\Vert\varphi^{(1)}\Vert^2_{H^{s-1}(\T)}\right)  .
    \end{equation}

\end{theorem}
Notice that the size of the existence time interval depends on the $H^3, H^2$ norms of the initial data, for all $s\ge3$. Every solution, with the regularity as in the statement of Theorem \ref{th_esistenza}, has the additional regularity
\begin{equation*}
\varphi\in C^2(I_0; H^{s-2}(\mathbb T)),
\end{equation*}
following from equation \eqref{equ1} and suitable commutator estimates, see \cite{M-S-T:ONDE2}, Corollary 16.


\subsection{The main result}

Now we state the main result of this paper about the continuous dependence in strong norm of solutions on the initial data.
%
%
%
%

\begin{theorem}\label{main}
Let $s\ge 3$ be a real number and $\mu>\delta>0$. Let us consider $\varphi^{(0)}\in H^{s}(\mathbb T)$, $\varphi^{(1)}\in H^{s-1}(\mathbb T)$ and sequences $\{\varphi^{(0)}_n\}_{n\in\N}\subset H^s(\T)$, $\{\varphi^{(1)}_n\}_{n\in\N}\subset H^{s-1}(\T)$, where all functions $\varphi^{(0)},\varphi^{(1)},\varphi^{(0)}_n,\varphi^{(1)}_n$ have zero spatial mean. Assume that 
$$\varphi^{(0)}_n\to\varphi^{(0)} \quad\mbox{strongly in }H^s(\T), \qquad
\varphi^{(1)}_n\to\varphi^{(1)} \quad\mbox{strongly in }H^{s-1}(\T),\;{\rm as }\; n\to+\infty.
$$ 

Let $T>0$ and set $I=[0,T]$. 
Assume that there exists a unique solution $\varphi\in C(I;H^s(\T))\cap C^1(I;H^{s-1}(\T))$ of problem \eqref{equ1}, \eqref{id} with initial data $\varphi^{(0)},\varphi^{(1)}$, and that for all $n$ there exist unique solutions $\varphi_n\in C(I;H^s(\T))\cap C^1(I;H^{s-1}(\T))$ of \eqref{equ1}, \eqref{id} with initial data $\varphi^{(0)}_n,\varphi^{(1)}_n$. All solutions satisfy \eqref{stab1}, \eqref{stima_Hs} on $I$ (with corresponding initial data in the r.h.s.) for a given constant $C_2$ independent of $n$.

Then
\begin{equation}
\begin{array}{ll}\label{}
\varphi_n\to\varphi \qquad {\rm strongly\; in }\;	\; C(I;H^s(\T))\cap C^1(I;H^{s-1}(\T)),\;{\rm as }\; n\to+\infty.
\end{array}
\end{equation}

\end{theorem}

Here we are assuming that all solutions $\vphi,\vphi_n$ are defined on the same time interval $I=[0,T]$, where $T$ is arbitrarily given, neither necessarily small nor given by \eqref{tmax}.

Using the a priori estimate \eqref{stima_Hs} and standard arguments, it is rather easy to show the continuous dependence of solutions on the initial data in the topology of $C(I;H^{s-\eps}(\T))\cap C^1(I;H^{s-1-\eps}(\T))$, for all small enough $\eps>0$. Instead, in Theorem \ref{main} we prove the continuous dependence precisely in the topology of $C(I;H^s(\T))\cap C^1(I;H^{s-1}(\T))$, i.e. in the same function space where we show the existence.
From Theorem \ref{th_esistenza} and Theorem \ref{main} we obtain that the initial value problem \eqref{equ1}, \eqref{id} is well-posed in $H^s$ in the classical sense of Hadamard.

The rest of the paper is organized as follows. In Section \ref{prbt} we introduce some notation and give some preliminary technical results. In Section \ref{proof main} we prove our main Theorem \ref{main}.

\section{Notations and preliminary results}\label{prbt}
\subsection{Notations}

In the paper we denote by $C$ generic positive constants, that may vary from line to line or even inside the same formula.

Let $\mathbb T$ denote the one-dimensional torus defined as
$\mathbb T:=\mathbb R/(2\pi\mathbb Z).$
As usual, all functions defined on $\mathbb T$ can be considered as $2\pi-$periodic functions on $\mathbb R$. 

All functions $f:\mathbb T\rightarrow\mathbb C$ can be expanded in terms of Fourier series as
\begin{equation*}
f(x)=\sum\limits_{k\in\mathbb Z}\widehat{f}(k)e^{ikx}\,,
\end{equation*}
where $\left\{\widehat{f}(k)\right\}_{k\in\mathbb Z}$ are the Fourier coefficients defined by
\begin{equation}\label{coeff_fourier}
\widehat{f}(k):=\frac1{2\pi}\int_{\mathbb T}f(x)e^{-ikx}\,dx\,,\qquad k\in\mathbb Z\,.
\end{equation}

For positive real numbers $s$, $H^s=H^s(\mathbb T)$ denotes the Sobolev space of order $s$ on $\mathbb T$, defined to be the set of functions $f:\mathbb T\rightarrow\mathbb C$ such that
\begin{equation}\label{normaHs}
\Vert f\Vert_{H^s}^2:=\sum\limits_{k\in\mathbb Z}\langle k\rangle^{2s}\vert\widehat{f}(k)\vert^2<+\infty\,,
\end{equation}
where it is set
\begin{equation}\label{bracket}
\langle k\rangle:=(1+\vert k\vert^2)^{1/2}\,.
\end{equation}
The function $\Vert\cdot\Vert_{H^s}$ defines a norm on $H^s$, associated to the inner product
\begin{equation*}
(f,g)_{H^s}:=\sum\limits_{k\in\mathbb Z}\langle k\rangle^{2s}\widehat{f}(k)\overline{\widehat{g}(k)}\,,
\end{equation*}
which turns $H^s$ into a Hilbert space.
For $s=0$ one has $H^0=L^2$. The $L^2$ norm will simply be denoted by $\|\cdot\|$.
From \eqref{normaHs} we have
\[
\Vert f\Vert_{H^s}= \| \langle \p_x\rangle^sf\|,
\]
where $\langle\partial_x\rangle^s$ is the Fourier multiplier of symbol $\langle k\rangle^s$, defined by
\begin{equation*}\label{ds}
\widehat{\langle\partial_x\rangle^s u}(k)=\langle k\rangle^s \widehat{u}(k)\,,\quad\forall\,k\in\mathbb Z\,.
\end{equation*} 
We will work with functions with zero spatial mean, so that $\widehat{f}(0)=0$. For these functions, we easily obtain from \eqref{normaHs} the Poincar\'e inequality
\begin{equation}
\begin{array}{ll}\label{poincare}

\Vert f\Vert_{H^s}\le \sqrt2\Vert f_x\Vert_{H^{s-1}} \qquad \forall s\ge1.
\end{array}
\end{equation}
For $T>0$ and $j\in\mathbb N$, we denote by $C^j([0,T]; \mathcal X)$ the space of $j$ times continuously differentiable functions $f:\mathbb R\rightarrow \mathcal X$.

\subsection{Preliminary results}

Let us consider the Cauchy problem
\begin{equation}\label{equ2}
\begin{cases}
\psi_{tt}-\left(\mu-2\phi_x\right)\psi_{xx}=F,   &\qquad t\in I, \; x\in \T,
\\
\psi_{|t=0}=\psi^{(0)}, \quad \p_t\psi_{|t=0}=\psi^{(1)}&\qquad  x\in \T,
\end{cases}
\end{equation}
with unknown the scalar function $\psi=\psi(t,x)$, and where $\phi=\mathbb H[\varphi]$ is the Hilbert transform of a given function $\varphi$, sufficiently smooth, with zero mean
 and such that
\begin{equation}\label{stab2}
\mu -2\phi_x\ge \delta \qquad t\in I,\; x\in \T,
\end{equation}
for given constants $\mu>\d>0$. Let us define
\[
\Ec(t):=\left( \|\psi_t(t)\|^2 + \int_\T(\mu -2\phi_x)|\psi_x(t)|^2  dx \right)^{1/2}.
\]

\begin{lemma}\label{lemmapsi0}
Let $\varphi\in C(I;H^3)\cap C^1(I;H^{2})$ be given and satisfying \eqref{stab2}. For all $\psi^{(0)}\in H^1$, $\psi^{(1)}\in L^2$, both functions with zero mean, and $F\in L^1(I;L^2)$ there exists a unique solution 
 $\psi\in C(I;H^1)\cap C^1(I;L^{2})$, of \eqref{equ2}, with zero mean and such that
\begin{equation}
\begin{array}{ll}\label{stimapsi0}
\ds \frac{d}{dt} \Ec
\le C \left( \|\varphi_t\|_{H^2} + \|\varphi\|_{H^3}  \right)\Ec + \| F\| ,
\end{array}
\end{equation}
for every $t\in I$.
\end{lemma}
\begin{proof}
To obtain \eqref{stimapsi0} we multiply the equation in \eqref{equ2} by $\psi_t$ and integrate over $\T$. Integrating by parts gives
\begin{equation*}
\begin{array}{ll}\label{}
\ds \frac12\frac{d}{dt} \left( \|\psi_t\|^2 + \int_\T(\mu -2\phi_x)|\psi_x|^2 dx \right) 
=2\int_\T\phi_{xx}\psi_t\psi_x\, dx -\int_\T\phi_{xt}|\psi_x|^2\, dx + \int_\T F\psi_t\, dx.
\end{array}
\end{equation*}
Then, by the Cauchy-Schwarz inequality, the Sobolev imbedding $H^1\hookrightarrow L^\infty$ and the estimate 
\begin{equation}\label{stima_hilbert}
\Vert\mathbb H[\varphi]\Vert_{H^s}\le\Vert \varphi\Vert_{H^s}\,,\qquad\forall\,\varphi\in H^s,\;  s\in\mathbb R,
\end{equation}
we obtain
\begin{equation}
\begin{array}{ll}\label{stimapsi00}
\ds \frac12 \frac{d}{dt} \left( \|\psi_t\|^2 + \int_\T(\mu -2\phi_x)|\psi_x|^2  dx \right)
\le C \left( \|\varphi_t\|_{H^2} + \|\varphi\|_{H^3}  \right)\left( \|\psi_t\|^2 + \|\psi_x\|^2  \right) + \| F\| \|\psi_t\|.
\end{array}
\end{equation}
From \eqref{stab2} we obtain the estimate
\begin{equation}\label{stimadeltapsix}
\|\psi_x\| \le \frac1{\sqrt\d}\left(  \int_\T(\mu -2\phi_x)|\psi_x|^2  dx \right)^{1/2} \le \frac1{\sqrt\d}\, \Ec,
\end{equation}
and, substituting it in \eqref{stimapsi00}, we obtain the bound
\begin{multline*}
\ds \frac12 \frac{d}{dt} \left( \|\psi_t\|^2 + \int_\T(\mu -2\phi_x)|\psi_x|^2  dx \right)
\\
\le C \left( \|\varphi_t\|_{H^2} + \|\varphi\|_{H^3}  \right)\left( \|\psi_t\|^2 + \int_\T(\mu -2\phi_x)|\psi_x|^2  dx   \right) + \| F\| \|\psi_t\|,
\end{multline*}
with a new constant $C$ also depending on $\d$. Since $\|\psi_t\|\le\Ec,$ the above inequality implies
\begin{equation}\label{stimapsi000}
\ds \frac{d}{dt} \Ec
\le C \left( \|\varphi_t\|_{H^2} + \|\varphi\|_{H^3}  \right)\Ec + \| F\| ,
\end{equation}
that is \eqref{stimapsi0}.
Applying the Gronwall lemma to \eqref{stimapsi000} gives the a priori estimate
\begin{equation*}\label{}
\Ec(t)
 \le e^{C T\max_{\tau\in I}\left( \|\varphi_t\|_{H^2} + \|\varphi\|_{H^3}  \right)} \left\{ \Ec(0) + \int^t_0\| F\|\, d\tau
\right\},
\end{equation*}
which yields
\begin{multline}\label{apL2}
\left( \|\psi_t(t)\|^2 + \|\psi_x(t)\|^2  \right)^{1/2}
\\
 \le Ce^{C T\max_{\tau\in I}\left( \|\varphi_t\|_{H^2} + \|\varphi\|_{H^3}  \right)} \left\{ \left( \|\psi^{(1)}\|^2 + (\mu +2\|\varphi^{(0)}\|_{H^2})\|\psi_x^{(0)}\|^2  \right)^{1/2} + \int^T_0\| F\|\, d\tau
\right\},
\end{multline}
for all $t\in I$, with $C$ also depending on $\d$. Using \eqref{apL2}, the existence of one solution is obtained by standard arguments, e.g. see \cite{kato}. 
By linearity of the problem, the difference of any two solutions with the same data satisfies \eqref{apL2} with zero right-hand side. This shows that the solution is defined up to an additive constant. Thus, requiring the solution to have zero mean gives one uniquely defined solution.
Since such a solution has zero mean, we can apply the Poincar\'e inequality \eqref{poincare} and obtain from \eqref{apL2}
\begin{multline*}\label{}
\left( \|\psi_t(t)\|^2 + \|\psi(t)\|^2_{H^1}  \right)^{1/2}
\\
 \le Ce^{C T\max_{\tau\in I}\left( \|\varphi_t\|_{H^2} + \|\varphi\|_{H^3}  \right)} \left\{ \left( \|\psi^{(1)}\|^2 + (\mu +2\|\varphi^{(0)}\|_{H^2})\|\psi_x^{(0)}\|^2  \right)^{1/2} + \int^T_0\| F\|\, d\tau
\right\}.
\end{multline*}
for all $t\in I$, which shows that $\psi \in C(I;H^1)\cap C^1(I;L^{2})$.
\end{proof}
The next lemma concerns the regularity of the solution $\psi$ to \eqref{equ2}.

\begin{lemma}\label{lemmapsi01}
Let $r\ge2$ and $s\ge\max\{3,r\}$. Let $\varphi\in C(I;H^s)\cap C^1(I;H^{2})$ satisfying \eqref{stab2}. For all $\psi^{(0)}\in H^r$, $\psi^{(1)}\in H^{r-1}$ with zero mean, and $F\in L^2(I;H^{r-1})$ there exists a unique solution 
 $\psi\in C(I;H^r)\cap C^1(I;H^{r-1})$ of \eqref{equ2},  with zero mean and such that
\begin{multline}\label{stimapsi01}
 \|\psi(t)\|^2_{H^{r}} +  \|\psi_t(t)\|^2_{H^{r-1}} 
\\
 \le Ce^{C T\max_{\tau\in I}\left( 1+ \|\varphi_t\|_{H^2} + \|\varphi\|_{H^{s}}  \right)}  \left\{  \|\psi^{(1)}\|^2_{H^{r-1}} + (\mu +2\|\varphi^{(0)}\|_{H^2})\|\psi^{(0)}\|^2_{H^{r}}   +  \|F  \|_{L^2(0,t;H^{r-1})}^2
\right\}
\end{multline}
for every $t\in I$.
\end{lemma}
\begin{proof}
We apply $\langle\partial_x\rangle^{r-1}$ to the equation in \eqref{equ2}, multiply by $\langle\partial_x\rangle^{r-1}\psi_t$ and integrate over $\T$. Integrating by parts gives
\begin{equation}
\begin{array}{ll}\label{equpsi01}
\ds \frac12\frac{d}{dt} \left( \|\psi_t\|^2_{H^{r-1}} + \int_\T(\mu -2\phi_x)|\langle\partial_x\rangle^{r-1}\psi_x|^2 dx \right) 
\\
\ds =2\int_\T\phi_{xx}\langle\partial_x\rangle^{r-1}\psi_t \, \langle\partial_x\rangle^{r-1}\psi_x\, dx -\int_\T\phi_{xt}|\langle\partial_x\rangle^{r-1}\psi_x|^2\, dx
\\
\ds -2\int_\T [\langle\partial_x\rangle^{r-1};\phi_x]\psi_{xx}\, \langle\partial_x\rangle^{r-1}\psi_t \, dx
 + \int_\T \langle\partial_x\rangle^{r-1}F \, \langle\partial_x\rangle^{r-1}\psi_t\, dx
 \\
=\ds \sum_{k=1}^4I_k.

\end{array}
\end{equation}
We estimate each term of this sum.

\noindent
{\it Estimate of $I_1$.} We apply the Cauchy-Schwarz inequality, the Sobolev imbedding $H^1\hookrightarrow L^\infty$ and the estimate 
\eqref{stima_hilbert}:
\begin{equation}
\begin{array}{ll}\label{i1}
|I_1 | \le 2 \|\phi_{xx} \|_{L^\infty} \| \langle\partial_x\rangle^{r-1}\psi_t\| \| \langle\partial_x\rangle^{r-1}\psi_x\|  \le C  \|\varphi _x\|_{H^2}\|\psi_t\|_{H^{r-1}}  \|\psi_x\|_{H^{r-1}} .
\end{array}
\end{equation}
{\it Estimate of $I_2$.} In a similar way we obtain
\begin{equation}
\begin{array}{ll}\label{i2}
|I_2| \le  \|\phi_{xt} \|_{L^\infty}  \| \langle\partial_x\rangle^{r-1}\psi_x\|^2
\le  C  \|\varphi _t\|_{H^2} \|\psi_x\|^2_{H^{r-1}} .
\end{array}
\end{equation}
{\it Estimate of $I_3$.} First of all we write 
\begin{equation}
\begin{array}{ll}\label{I30}
|I_3| \le  2\|  [\langle\partial_x\rangle^{r-1};\phi_x]\psi_{xx} \| \, \| \langle\partial_x\rangle^{r-1}\psi_t \| \, .
\end{array}
\end{equation}
For the estimate of the commutator we need to distinguish between the different values of $r$.

i) If $r>5/2$ we apply the estimate \eqref{stima_comm_ds}
 for commutators with $\t=r-1, \sigma=r-2>1/2$ and \eqref{stima_hilbert} to obtain
 \begin{equation}
\begin{array}{ll}\label{I3i}
\|  [\langle\partial_x\rangle^{r-1};\phi_x]\psi_{xx} \| \le  C \left( \|\phi_x\|_{H^{r-1}}  + \|\phi_{xx}\|_{H^{1}}
\right)  \|\psi_{xx}\|_{H^{r-2}} 
 \\
\le  C \left( \|\varphi_x\|_{H^{r-1}} + \|\varphi_{x}\|_{H^{2}} \right) \|\psi_{x}\|_{H^{r-1}} 

\le  C \|\varphi_x\|_{H^{s-1}} \|\psi_{x}\|_{H^{r-1}} .

\end{array}
\end{equation}
ii) If $2\le r<5/2$ we apply \eqref{stima_comm_ds1} with $\t=r-1, \sigma=3-r>1/2$ and \eqref{stima_hilbert} to get
 \begin{equation}
\begin{array}{ll}\label{I3ii}
\|  [\langle\partial_x\rangle^{r-1};\phi_x]\psi_{xx} \| \le  C \left( \|\phi_x\|_{H^{2}}  \|\psi_{xx}\| + \|\phi_{xx}\|_{H^{1}}
 \|\psi_{xx}\|_{H^{r-2}} \right) 
 \\
\le  C  \|\varphi_{x}\|_{H^{2}} \left( \|\psi_{x}\|_{H^{1}}  +  \|\psi_{x}\|_{H^{r-1}}  \right) 

\le  C \|\varphi_x\|_{H^{s-1}} \|\psi_{x}\|_{H^{r-1}} .

\end{array}
\end{equation}
iii) Finally if $ r=5/2$ we apply \eqref{stima_comm_ds2} with $\t=r-1$ and \eqref{stima_hilbert} to obtain
 \begin{equation}
\begin{array}{ll}\label{I3iii}
\|  [\langle\partial_x\rangle^{r-1};\phi_x]\psi_{xx} \| \le  C \left( \|\phi_x\|_{H^{2}}  \|\psi_{xx}\|_{H^{1/2}} + \|\phi_{xx}\|_{H^{1}}
 \|\psi_{xx}\|_{H^{r-2}} \right) 
 \\
\le  C  \|\varphi_{x}\|_{H^{2}}  \|\psi_{x}\|_{H^{3/2}}  

\le  C \|\varphi_x\|_{H^{s-1}} \|\psi_{x}\|_{H^{r-1}} .

\end{array}
\end{equation}
Recalling \eqref{I30}, from \eqref{I3i}--\eqref{I3iii} we have obtained, for all values of $r\ge2$,
\begin{equation}
\begin{array}{ll}\label{i3}
|I_3| \le  C \|\varphi_x\|_{H^{s-1}} \|\psi_{x}\|_{H^{r-1}} \|\psi_t\|_{H^{r-1}}  .
\end{array}
\end{equation}
{\it Estimate of $I_4$.} The Cauchy-Schwarz inequality gives
\begin{equation}
\begin{array}{ll}\label{i4}
|I_4 | \le  \| \langle\partial_x\rangle^{r-1}F \|  \|\langle\partial_x\rangle^{r-1}\psi_t\|
=  \| F\|_{H^{r-1}} \|\psi_t\|_{H^{r-1}} .
\end{array}
\end{equation}
Estimating the right-hand side of \eqref{equpsi01} by \eqref{i1}, \eqref{i2}, \eqref{i3}, \eqref{i4} yields 
\begin{multline}\label{27}
\ds \frac{d}{dt} \left( \|\psi_t\|^2_{H^{r-1}} + \int_\T(\mu -2\phi_x)|\langle\partial_x\rangle^{r-1}\psi_x|^2  dx \right)
\\
\ds \le C \left( \|\varphi_t\|_{H^2} + \|\varphi_x\|_{H^{s-1}}  \right)\left( \|\psi_t\|^2_{H^{r-1}} + \|\psi_x\|^2_{H^{r-1}}  \right) 
\ds    + 2\|F\|_{H^{r-1}} \|\psi_t\|_{H^{r-1}} 
\\
\ds \le C \left(1+ \|\varphi_t\|_{H^2} + \|\varphi_x\|_{H^{s-1}}  \right)\left( \|\psi_t\|^2_{H^{r-1}} + \|\psi_x\|^2_{H^{r-1}}  \right) 
   + \|F\|_{H^{r-1}} ^2,
\\
\ds \le C \left( 1+ \|\varphi_t\|_{H^2} + \|\varphi_x\|_{H^{s-1}}  \right)\left( \|\psi_t\|^2_{H^{r-1}} +\int_\T(\mu -2\phi_x)|\langle\partial_x\rangle^{r-1}\psi_x|^2  dx   \right) 
   + \|F\|_{H^{r-1}} ^2,
\end{multline}
where in the last inequality we have used \eqref{stab2}. Applying the Gronwall lemma to \eqref{27} gives
\begin{multline*}\label{}
 \|\psi_t(t)\|^2_{H^{r-1}}  + \int_\T(\mu -2\phi_x)|\langle\partial_x\rangle^{r-1}\psi_x(t)|^2 dx 
\\
 \le e^{C T\max_{\tau\in I}\left( 1+ \|\varphi_t\|_{H^2} + \|\varphi_x\|_{H^{s-1}}  \right)} \left\{  \|\psi^{(1)}\|^2_{H^{r-1}} + \int_\T(\mu -2\phi_x(0))|\langle\partial_x\rangle^{r-1}\psi_x^{(0)}|^2dx   + \int^t_0\| F\|^2_{H^{r-1}}\, d\tau
\right\},
\end{multline*}
which yields, by \eqref{stab2} and the Poincar\'e inequality,
\begin{multline}\label{apHs}
 \|\psi_t(t)\|^2_{H^{r-1}} + \|\psi(t)\|^2 _{H^{r}} 
\\
 \le Ce^{C T\max_{\tau\in I}\left( 1+ \|\varphi_t\|_{H^2} + \|\varphi\|_{H^{s}}  \right)} \left\{  \|\psi^{(1)}\|^2_{H^{r-1}} + (\mu +2\|\varphi^{(0)}\|_{H^2})\|\psi^{(0)}\|^2_{H^{r}}   + \int^t_0\| F\|^2_{H^{r-1}}\, d\tau
\right\},
\end{multline}
for all $t\in I$, with $C$ also depending on $\d$, that is \eqref{stimapsi01}.
 \eqref{apHs} provides the a priori estimate for $\psi$ in  $ C(I;H^r)\cap C^1(I;H^{r-1})$.
\end{proof}


Now we consider the quadratic operator $\mathcal Q\left[\varphi\right]$, defined in \eqref{termine_nl}. First of all we recall the estimate proved in \cite{M-S-T:ONDE2}.
\begin{proposition}\label{lemma_stima_quadr}
There exists a positive constant $C$ such that for every real $s\ge 1$ and for all $\varphi\in H^s\cap H^3$
\begin{equation}\label{stima_quadr}
\Vert \mathcal Q[\varphi]\Vert_{H^{s-1}}\le C\Vert\varphi_x\Vert_{H^2}\Vert\varphi_x\Vert_{H^{s-1}}\,.
\end{equation}

\end{proposition}
\begin{proof}
For the proof see \cite{M-S-T:ONDE2}.
\end{proof}

Next we give an estimate for the difference of values of $\mathcal Q\left[\varphi\right]$.
\begin{lemma}\label{}
The quadratic operator $\mathcal Q\left[\varphi\right]$, defined in \eqref{termine_nl}, satisfies
\begin{equation}
\begin{array}{ll}\label{diffQ}
\| \mathcal Q\left[\varphi \right] - \mathcal Q\left[\tilde\varphi \right] \| \le C \left( \|\varphi\|_{H^3}+ \|\tilde\varphi\|_{H^3} \right) \|(\varphi-\tilde\varphi)_x\|
\end{array}
\end{equation}
for all functions $\varphi,\tilde\varphi \in H^3.$
\end{lemma}
\begin{proof}
We denote $\phi=\mathbb H[\varphi]$, $\tilde\phi=\mathbb H[\tilde\varphi]$, $\d\varphi=\varphi-\tilde\varphi$, $\d\phi=\phi-\tilde\phi$. Then we can write
\begin{equation*}
\begin{array}{ll}\label{}
\mathcal Q\left[\varphi \right] - \mathcal Q\left[\tilde\varphi \right] = -3\left[\mathbb H\,;\,\phi_x\right]\phi_{xx}-\left[\mathbb H\,;\,\phi\right]\phi_{xxx}
+3[\mathbb H\,;\,\tilde\phi_x]\tilde\phi_{xx}+[\mathbb H\,;\,\tilde\phi ]\tilde\phi_{xxx}
\\
=-3\left[\mathbb H\,;\,\d\phi_x \right]\phi_{xx}   
-3\left[\mathbb H\,;\,\tilde\phi_x \right]\d\phi_{xx}  -  \left[\mathbb H\,;\,\d\phi\right]\phi_{xxx}  -  [\mathbb H\,;\,\tilde\phi ]\d\phi_{xxx}
\\
=\ds \sum_{k=1}^4J_k.
\end{array}
\end{equation*}
Now we estimate each term of the sum.

\noindent
{\it Estimate of $J_1$.} We apply estimate \eqref{stima_comm_p}, with $p=s=0$, and \eqref{stima_hilbert}:
\begin{equation}
\begin{array}{ll}\label{j1}
\|J_1\| \le C \| \d\phi_x \| \|\phi_{xx} \|_{H^1} \le C \| \d\varphi_x \| \|\varphi \|_{H^3}.
\end{array}
\end{equation}
{\it Estimate of $J_2$.} We apply estimate \eqref{stima_comm_p}, with $p=2, s=0$, \eqref{stima_hilbert} and the Poincar\'e inequality:
\begin{equation}
\begin{array}{ll}\label{j2}
\|J_2\| \le C \| \tilde\phi_{xxx} \| \|\d\phi \|_{H^1} \le C \|\tilde\varphi \|_{H^3} \| \d\varphi_x \| .
\end{array}
\end{equation}
{\it Estimate of $J_3$.} We apply estimate \eqref{stima_comm_1}, with $ s=1$, \eqref{stima_hilbert} and the Poincar\'e inequality:
\begin{equation}
\begin{array}{ll}\label{j3}
\|J_3\| \le C \| \d\phi \|_{H^1} \|\phi_{xxx} \| \le C \|\d\varphi_x \| \| \varphi \|_{H^3} .
\end{array}
\end{equation}
{\it Estimate of $J_4$.} We apply estimate \eqref{stima_comm_p}, with $p=3, s=0$, \eqref{stima_hilbert} and the Poincar\'e inequality:
\begin{equation}
\begin{array}{ll}\label{j4}
\|J_4\| \le C \| \tilde\phi_{xxx} \| \|\d\phi \|_{H^1} \le C \|\tilde\varphi \|_{H^3} \| \d\varphi_x \| .
\end{array}
\end{equation}
Collecting \eqref{j1}--\eqref{j4} gives \eqref{diffQ}.
 
\end{proof}
More generally we apply \eqref{stima_comm_p} with different choices of $p$ to estimate the $H^{s-1}$-norm of $J_1-J_4$ and obtain
\begin{lemma}\label{}
Let $s\ge3$. The quadratic operator $\mathcal Q\left[\varphi\right]$, defined in \eqref{termine_nl}, satisfies
\begin{equation}
\begin{array}{ll}\label{diffQs}
\| \mathcal Q\left[\varphi \right] - \mathcal Q\left[\tilde\varphi \right] \|_{H^{s-1}} \le C \left( \|\varphi\|_{H^s}+ \|\tilde\varphi\|_{H^s} \right) \| \varphi-\tilde\varphi\|_{H^s}
\end{array}
\end{equation}
for all functions $\varphi,\tilde\varphi \in H^s.$
\end{lemma}


\section{Proof of Theorem \ref{main}}\label{proof main}

By assumption the sequence $\{\varphi^{(0)}_n\}_{n\in\N}$ is uniformly bounded in $H^s$, and the sequence $\{\varphi^{(1)}_n\}_{n\in\N}$ is uniformly bounded in $H^{s-1}$. Because of the uniform a priori estimate \eqref{stima_Hs},
from now on we may assume that the sequence of solutions $\{\varphi_n\}_{n\in\N}$ is uniformly bounded in $C(I;H^s)\cap C^1(I;H^{s-1})$ on the common time interval $I=[0,T]$, i.e. there exists $K>0$ such that
\begin{equation}
\begin{array}{ll}\label{stimapsin}
\|\varphi_n(t)\|_{H^s}^2+\|(\varphi_n)_t(t)\|_{H^{s-1}}^2\le C_2\left( \|\varphi_n^{(0)}\|_{H^s}^2+\|\varphi_n^{(1)}\|_{H^{s-1}}^2 \right) \le K, \qquad \forall t\in I, \, \forall n.
\end{array}
\end{equation}
Notice that the similar bound holds as well for the solution $\varphi$.

From now on $T$ will be usually included in the generic constant $C$, but sometimes not, when we prefer to emphasize its presence.


In the next proposition we prove the convergence of $\{\varphi_n\}_{n\in\N}$ in a weaker topology than in our main Theorem \ref{main}.
\begin{proposition}\label{convs-1}
Let $s\ge3$. Under the assumptions of Theorem \ref{main} the sequence of solutions $\{\varphi_n\}_{n\in\N}$ converges to $\varphi$ strongly in $C(I;H^{s-1})\cap C^1(I;H^{s-2})$.
\end{proposition}
\begin{proof}
We take the difference of equation \eqref{equ1ter} for $\varphi$ and $\varphi_n$ and get \[
(\varphi-\varphi_n)_{tt} -\left(\mu-2\phi_x\right)(\varphi-\varphi_n)_{xx}= - \mathcal Q\left[\varphi\right]  + \mathcal Q\left[\varphi_n \right] +
2\left(\phi_{n,x}-\phi_{x}\right)\varphi_{n,xx}\, ,
\]
(where $\phi_{n,x}=(\phi_n)_{x}, \varphi_{n,xx}=(\varphi_n)_{xx}$)
which has the form of \eqref{equ2} with
\[
\psi =\varphi-\varphi_n,
\]
\[
F=- \mathcal Q\left[\varphi\right] + \mathcal Q\left[\varphi_n \right]  +
2\left(\phi_{n,x}-\phi_{x}\right)\varphi_{n,xx}.
\] 
Applying \eqref{diffQ} gives
\begin{equation}
\begin{array}{ll}\label{f1}
\| \mathcal Q\left[\varphi \right] - \mathcal Q\left[\varphi_n \right] \| \le C  \left( \|\varphi\|_{H^3}+ \|\varphi_n\|_{H^3} \right) \|\psi_x\|
.
\end{array}
\end{equation}
We also have
\begin{equation}
\begin{array}{ll}\label{f2}
\| 2\left( \phi_{n,x}-\phi_{x}\right)\varphi_{n,xx} \| \le 2 \| \phi_{n,x}-\phi_{x} \| \| \varphi_{n,xx} \|_{L^\infty} \le C\|\psi_x\| \|\varphi_n \|_{H^3}.
\end{array}
\end{equation}
Thus, from \eqref{f1}, \eqref{f2} we obtain
\begin{equation}
\begin{array}{ll}\label{stimaF}
\| F\| \le C  \left( \|\varphi\|_{H^3}+ \|\varphi_n\|_{H^3} \right) \|\psi_x\|
.
\end{array}
\end{equation}
From Lemma \ref{lemmapsi0}, \eqref{stimadeltapsix}, \eqref{stimaF} we get
\begin{equation}\label{stimapsi1}
\ds \frac{d}{dt} \Ec
\le C \left( \|\varphi_t\|_{H^2} +\|\varphi\|_{H^3}  \right)\Ec + C  \left( \|\varphi\|_{H^3}+ \|\varphi_n\|_{H^3} \right) \|\psi_x\|
\ds \le C \Ec ,
\end{equation}
where we have used the uniform boundedness \eqref{stimapsin} for $\varphi_n$ and $\varphi$.
Applying Gronwall's lemma to \eqref{stimapsi1} and using \eqref{stab1} again yields
\[
\|\psi_t(t)\|^2 + \|\psi_x(t)\|^2   \le Ce^{CT}\left( \|\psi_t(0)\|^2 + \|\psi_x(0)\|^2  \right)  \qquad t\in I,
\]
that is
\[
 \|(\varphi-\varphi_n)_t(t)\|^2 + \|(\varphi-\varphi_n)_x(t)\|^2   \le Ce^{CT}\left( \|\varphi^{(1)}-\varphi^{(1)}_n\|^2 + \|(\varphi^{(0)}-\varphi^{(0)}_n)_x\|^2  \right)  \qquad t\in I,
\]
which gives the strong convergence of $\varphi_n$ to $\varphi$ in $C(I;H^1)\cap C^1(I;L^2)$, when passing to the limit as $n\to+\infty$. Recall that, since we are working with functions with zero spatial mean, the Poincar\'e inequality \eqref{poincare} holds. By interpolation and the uniform boundedness \eqref{stimapsin} we get
\begin{multline*}
 \|(\varphi-\varphi_n)_t(t)\|^2_{H^{s-2}} + \|(\varphi-\varphi_n)(t)\|^2_{H^{s-1}}
 \\
  \le C \left( \|(\varphi-\varphi_n)_t(t)\|_{H^{s-1}}^{1-1/(s-1)} \|(\varphi-\varphi_n)_t(t)\|_{L^2}^{1/(s-1)} + \|(\varphi-\varphi_n)(t)\|_{H^{s}}^{1-1/(s-1)} \|(\varphi-\varphi_n)(t)\|_{H^1}^{1/(s-1)} \right)^2
  \\
  \le C \left( \|(\varphi-\varphi_n)_t(t)\|_{L^2}^{2/(s-1)} +  \|(\varphi-\varphi_n)(t)\|_{H^1}^{2/(s-1)} \right), \qquad t\in I.
\end{multline*}
Finally, passing to the limit as $n\to+\infty$ in the above inequality gives the thesis.
\end{proof}


\begin{remark}\label{}
Obviously, by a similar argument with a finer interpolation we could prove the strong convergence of $\varphi_n$ to $\varphi$ in $C(I;H^{s-\eps})\cap C^1(I;H^{s-1-\eps})$, for all small enough $\eps>0$. However, this is useless for the following argument.
\end{remark}

Now we take one spatial derivative of \eqref{equ1ter} and get 
\[
(\varphi_x)_{tt} -\left(\mu-2\phi_x\right)(\varphi_x)_{xx}= - \mathcal Q\left[\varphi\right]_x  -
2 \phi_{xx}\varphi_{xx}\, ,
\]
which has the form of \eqref{equ2} with
\[
\psi =\varphi_x,
\qquad
F= - \mathcal Q\left[\varphi\right]_x  -
2 \phi_{xx}\varphi_{xx}.
\] 
Using Proposition \ref{lemma_stima_quadr}, a Moser-type estimate, the Sobolev imbedding and \eqref{stima_hilbert} we compute 
\begin{equation}
\begin{array}{ll}\label{}
\|F\|_{H^{s-2}} \le \| \mathcal Q\left[\varphi\right]\|_{H^{s-1}} + 2 \|\phi_{xx}\varphi_{xx}\|_{H^{s-2}}
\\
\le C \left(\|\varphi_x\|_{H^{2}}\|\varphi_x\|_{H^{s-1}} +  \|\phi_{xx}\|_{L^\infty}\|\varphi_{xx}\|_{H^{s-2}} +  \|\phi_{xx}\|_{H^{s-2}} \|\varphi_{xx}\|_{L^\infty} \right)  
\\
\le C \|\varphi\|_{H^{3}}\|\varphi\|_{H^{s}}  ,

\end{array}
\end{equation}
and we deduce that $F\in L^\infty(I;H^{s-2})$. Moreover, the initial values of $\varphi_x,\, (\varphi_x)_t$ are $\varphi_{x}^{(0)} \in H^{s-1},\, \varphi_{x}^{(1)}\in H^{s-2}$, respectively.

We are going to introduce regularized approximations of the data $\varphi_{x}^{(0)},\, \varphi_{x}^{(1)},\, F$.
Given any $\eps>0$, let us take functions $ \Psi^{(0)}_\eps\in H^{s}, \Psi^{(1)}_\eps\in H^{s-1}$ with zero mean, and $F^\eps \in L^2(I;H^{s-1})$ such that
\begin{equation}
\begin{array}{ll}\label{stimaeps}
\| \Psi^{(0)}_\eps - \varphi_{x}^{(0)}\|_{H^{s-1}} + \| \Psi^{(1)}_\eps - \varphi_{x}^{(1)}\|_{H^{s-2}} + \|F^\eps - F \|_{L^2(I;H^{s-2})} \le \eps.
\end{array}
\end{equation}

Let us consider the Cauchy problem with regularized data
\begin{equation}\label{equ3}
\begin{cases}
\Psi^\eps_{tt}-\left(\mu-2\phi_x\right)\Psi^\eps_{xx}=F^\eps,   &\qquad t\in I, \; x\in \T,
\\
\Psi^\eps_{|t=0}=\Psi^{(0)}_\eps, \quad \p_t\Psi^\eps_{|t=0}=\Psi^{(1)}_\eps&\qquad  x\in \T.
\end{cases}
\end{equation}
Again this problem has the form \eqref{equ2}.


\begin{lemma}\label{}
Let $s\ge3$. For any $\eps>0$, the Cauchy problem \eqref{equ3} has a unique solution $\Psi^\eps\in C(I;H^{s})\cap C^1(I;H^{s-1})$ and
\begin{multline}\label{stimaPsi}
 \|\Psi^\eps \|^2_{C(I;H^{s})}  +  \|\Psi^\eps _t\|^2_{C(I;H^{s-1})} 
\\
 \le Ce^{C T\max_{\tau\in I}\left( 1+ \|\varphi_t\|_{H^2} + \|\varphi\|_{H^{s}}  \right)}  \left\{  \|\Psi^{(1)}_\eps\|^2_{H^{s-1}} + (\mu +2\|\varphi^{(0)}\|_{H^2})\|\Psi^{(0)}_\eps\|^2_{H^{s}}   + \|F^\eps  \|_{L^2(I;H^{s-1})}^2
 \right\}.
\end{multline}
\end{lemma}
\begin{proof}
The result follows from Lemma \ref{lemmapsi01} with $r=s$.
\end{proof}
Now we estimate the difference $\Psi^\eps - \varphi_{x}$, which solves the linear problem
\begin{equation}
\begin{cases}\label{probdiff}
(\Psi^\eps - \varphi_{x})_{tt}-\left(\mu-2\phi_x\right)(\Psi^\eps - \varphi_{x})_{xx}=F^\eps-F,   &\qquad t\in I, \; x\in \T,
\\
(\Psi^\eps - \varphi_{x})_{|t=0}=\Psi^{(0)}_\eps- \varphi_{x}^{(0)}, \quad \p_t(\Psi^\eps - \varphi_{x})_{|t=0}=\Psi^{(1)}_\eps- \varphi_{x}^{(1)}&\qquad  x\in \T.
\end{cases}
\end{equation}
\begin{lemma}\label{}
Let $s\ge3$. For any $\eps>0$, the difference $\Psi^\eps - \varphi_{x}$ satisfies the estimate
\begin{equation}
\begin{array}{ll}\label{stimadiff}
\ds \|(\Psi^\eps  - \varphi_{x})(t) \|_{H^{s-1}} +  \| (\Psi^\eps  - \varphi_{x})_t(t) \|_{H^{s-2}} \le C\eps \qquad \forall t\in I.
\end{array}
\end{equation}

\end{lemma}
\begin{proof}
We apply Lemma \ref{lemmapsi01} with $r=s-1$ and get
\begin{multline}\label{stimadiff2}
 \|\Psi^\eps  - \varphi_{x} \|^2_{C(I;H^{s-1})} +  \|(\Psi^\eps  - \varphi_{x})_t \|^2_{C(I;H^{s-2})} 
\\
 \le Ce^{C T\max_{\tau\in I}\left( 1+ \|\varphi_t\|_{H^2} + \|\varphi\|_{H^{s}}  \right)}  \Big\{  \| \Psi^{(1)}_\eps- \varphi_{x}^{(1)}\|^2_{H^{s-2}} 
 \\
 + (\mu +2\|\varphi^{(0)}\|_{H^2})\| \Psi^{(0)}_\eps- \varphi_{x}^{(0)} \|^2_{H^{s-1}}   +  \|F^\eps-F  \|_{L^2(I;H^{s-2})}^2
\Big\}.
\end{multline}
Thus \eqref{stimadiff} follows from \eqref{stimaeps}.
\end{proof}

Finally we estimate the difference between $\varphi_{n,x}=(\varphi_n)_{x}$ and $\Psi^\eps$.
The difference of the corresponding problems reads
\begin{equation}
\begin{cases}\label{probdiff2}
(\Psi^\eps - \varphi_{n,x})_{tt}-\left(\mu-2\phi_{n,x}\right)(\Psi^\eps - \varphi_{n,x})_{xx}=G^{n,\eps},   &\qquad t\in I, \; x\in \T,
\\
(\Psi^\eps - \varphi_{n,x})_{|t=0}=\Psi^{(0)}_\eps- \varphi_{n,x}^{(0)}, \\ \p_t(\Psi^\eps - \varphi_{n,x})_{|t=0}=\Psi^{(1)}_\eps- \varphi_{n,x}^{(1)}&\qquad  x\in \T.
\end{cases}
\end{equation}
where we have set
\begin{equation}
\begin{array}{ll}\label{defGneps}

G^{n,\eps} =F^\eps-F^n + 2\left(\phi_{n,x}-\phi_{x}\right)\Psi^\eps_{xx}, \qquad F^n= - \mathcal Q\left[\varphi_n\right]_x  -
2 \phi_{n,xx}\varphi_{n,xx},
\end{array}
\end{equation}
\[
\varphi_{n,x}^{(0)}=(\varphi_{n}^{(0)})_x, \qquad \varphi_{n,x}^{(1)}=(\varphi_{n}^{(1)})_x .
\]
\begin{lemma}\label{}
Let $s\ge3$. For any $\eps>0$, there exists $M(\eps)>0$ such that, for any $n$ the difference $\Psi^\eps - \varphi_{n,x}$ satisfies the estimate
\begin{multline}\label{stimadiff6}
 \|(\Psi^\eps  - \varphi_{n,x})(t) \|^2_{H^{s-1}} +  \|(\Psi^\eps  - \varphi_{n,x})_t(t) \|^2_{H^{s-2}} 
\\
 \le C \Big\{ \eps^2 + \| \varphi^{(1)}- \varphi_{n}^{(1)}\|^2_{H^{s-1}} 
 + \| \varphi^{(0)} - \varphi_{n}^{(0)} \|^2_{H^{s}}   +  \left| \int_0^t\|\varphi-\varphi_n\|_{H^{s}}^2  d\tau \right| + TM(\eps)\|\varphi_n -\varphi\|^2_{C(I;H^{s-1})}
\Big\}
\end{multline}
for any $t\in I$.
\end{lemma}
%
%
\begin{proof}
We apply Lemma \ref{lemmapsi01} with $r=s-1$ and obtain for every $t\in I$
\begin{multline}\label{stimadiff4}
 \|(\Psi^\eps  - \varphi_{n,x})(t) \|^2_{H^{s-1}} +  \|(\Psi^\eps  - \varphi_{n,x})_t(t) \|^2_{H^{s-2}} 
\\
 \le Ce^{C T\max_{\tau\in I}\left( 1+ \|\varphi_{n,t}\|_{H^2} + \|\varphi_n\|_{H^{s}}  \right)}  \Big\{  \| \Psi^{(1)}_\eps- \varphi_{n,x}^{(1)}\|^2_{H^{s-2}} 
 \\
 + (\mu +2\|\varphi^{(0)}_n\|_{H^2})\| \Psi^{(0)}_\eps- \varphi_{n,x}^{(0)} \|^2_{H^{s-1}}   +  \|G^{n,\eps}   \|_{L^2(0,t;H^{s-2})}^2
\Big\}.
\end{multline}
First of all, recalling the uniform boundedness w.r.t. $n$ \eqref{stimapsin}, we may write \eqref{stimadiff4} as
\begin{multline}\label{stimadiff5}
 \|(\Psi^\eps  - \varphi_{n,x})(t) \|^2_{H^{s-1}} +  \|(\Psi^\eps  - \varphi_{n,x})_t(t) \|^2_{H^{s-2}} 
\\
 \le C \Big\{  \| \Psi^{(1)}_\eps- \varphi_{n,x}^{(1)}\|^2_{H^{s-2}} 
 + \| \Psi^{(0)}_\eps- \varphi_{n,x}^{(0)} \|^2_{H^{s-1}}   +  \|G^{n,\eps}   \|_{L^2(0,t;H^{s-2})}^2
\Big\}.
\end{multline}
From the definition in \eqref{defGneps} we have, for every $\tau\in[0,t]$,
\begin{equation*}
\begin{array}{ll}\label{}
\|G^{n,\eps} \|_{H^{s-2}} \le \| F^\eps-F^n
 \|_{H^{s-2}} + 2\|\left(\phi_{n,x}-\phi_{x}\right)\Psi^\eps_{xx} \|_{H^{s-2}}  \\
\\
\le  \| F^\eps-F
 \|_{H^{s-2}}
+ \| F-F^n
 \|_{H^{s-2}} + C\|\varphi_n -\varphi\|_{H^{s-1}} \|\Psi^\eps \|_{H^{s}} 
.
\end{array}
\end{equation*}
Integrating this inequality in $\tau$ between $0$ and $t$ gives
\begin{multline}\label{42}
\ds \left|  \int_0^t\|G^{n,\eps} \|^2_{H^{s-2}}  d\tau \right|
\le  3\left| \int_0^t \| F^\eps-F
 \|^2_{H^{s-2}} d\tau \right|
+ 3\left| \int_0^t\| F-F^n
 \|^2_{H^{s-2}} d\tau \right| + CTM(\eps)\|\varphi_n -\varphi\|^2_{C(I;H^{s-1})} 
\end{multline}
where we have denoted
\[
 \ds  M(\eps) := Ce^{C T\max_{\tau\in I}\left( 1+ \|\varphi_t\|_{H^2} + \|\varphi\|_{H^{s}}  \right)}  \left\{  \|\Psi^{(1)}_\eps\|^2_{H^{s-1}} + (\mu +2\|\varphi^{(0)}\|_{H^2})\|\Psi^{(0)}_\eps\|^2_{H^{s}}   + \|F^\eps  \|_{L^2(I;H^{s-1})}^2
 \right\},
\]
that is the right-hand side of \eqref{stimaPsi}. 
On the other hand, for all $\tau$ we have
\begin{equation}
\begin{array}{ll}\label{diffF1}
 \| F-F^n
 \|_{H^{s-2}} \leq  \|  \mathcal Q\left[\varphi\right] - \mathcal Q\left[\varphi_n\right] \|_{H^{s-1}} + 2 \| \phi_{xx}\varphi_{xx} -
 \phi_{n,xx}\varphi_{n,xx} \|_{H^{s-2}}.
\end{array}
\end{equation}
From \eqref{diffQs} we have
\begin{equation}
\begin{array}{ll}\label{}
\|  \mathcal Q\left[\varphi\right] - \mathcal Q\left[\varphi_n\right] \|_{H^{s-1}}
 \le C (\|\varphi\|_{H^{s}}+ \|\varphi_n\|_{H^{s}}) \|\varphi-\varphi_n\|_{H^{s}} 
.
\end{array}
\end{equation}
Moreover, since $H^{s-2}$ is an algebra we can estimate
\begin{multline}\label{diffF2}
2 \| \phi_{xx}\varphi_{xx} -
 \phi_{n,xx}\varphi_{n,xx} \|_{H^{s-2}}
\\ \le 
C \| \phi_{xx} -
 \phi_{n,xx} \|_{H^{s-2}} \| \varphi_{xx}\|_{H^{s-2}}
+ C \| \phi_{n,xx}\|_{H^{s-2}} \|\varphi_{xx} -
\varphi_{n,xx} \|_{H^{s-2}}
\\
 \le C (\|\varphi\|_{H^{s}}+ \|\varphi_n\|_{H^{s}}) \|\varphi-\varphi_n\|_{H^{s}} 
.
\end{multline}
From \eqref{diffF1}--\eqref{diffF2} and the uniform boundedness \eqref{stimapsin} we obtain 
\begin{equation}
\begin{array}{ll}\label{diffF}
 \| F-F^n
 \|_{H^{s-2}}  \le C  \|\varphi-\varphi_n\|_{H^{s}} 
\qquad \forall \tau\in I,
\end{array}
\end{equation}
and substituting it in \eqref{42} gives
\begin{multline}\label{stimaGneps}
\ds \left|  \int_0^t\|G^{n,\eps} \|^2_{H^{s-2}}  d\tau \right|
\le  3\left| \int_0^t \| F^\eps-F
 \|^2_{H^{s-2}} d\tau \right|
+ C\left| \int_0^t\|\varphi-\varphi_n\|_{H^{s}}^2  d\tau \right| + CTM(\eps)\|\varphi_n -\varphi\|^2_{C(I;H^{s-1})} .
\end{multline}
Finally, from \eqref{stimaeps}, \eqref{stimadiff5}, \eqref{stimaGneps} we get

\begin{multline*}\label{}
 \|(\Psi^\eps  - \varphi_{n,x})(t) \|^2_{H^{s-1}} +  \|(\Psi^\eps  - \varphi_{n,x})_t(t) \|^2_{H^{s-2}} 
\\
 \le C \Big\{ \eps^2 + \| \varphi^{(1)}- \varphi_{n}^{(1)}\|^2_{H^{s-1}} 
 + \| \varphi^{(0)} - \varphi_{n}^{(0)} \|^2_{H^{s}}   +  \left| \int_0^t\|\varphi-\varphi_n\|_{H^{s}}^2  d\tau \right| + TM(\eps)\|\varphi_n -\varphi\|^2_{C(I;H^{s-1})}
\Big\}
\end{multline*}
for all $t\in I$, that is \eqref{stimadiff6}.
\end{proof}


\begin{proof}[Proof of Theorem \ref{main}]
Adding \eqref{stimadiff}, \eqref{stimadiff6}, and applying the Poincar\'e inequality gives
\begin{multline*}\label{}
 \|(\varphi  - \varphi_{n})(t) \|^2_{H^{s}} +  \|(\varphi  - \varphi_{n})_t(t) \|^2_{H^{s-1}} 
\\
 \le C \Big\{ \eps^2 + \| \varphi^{(1)}- \varphi_{n}^{(1)}\|^2_{H^{s-1}} 
 + \| \varphi^{(0)} - \varphi_{n}^{(0)} \|^2_{H^{s}}   +  \left| \int_0^t\|\varphi-\varphi_n\|_{H^{s}}^2  d\tau \right| + TM(\eps)\|\varphi_n -\varphi\|^2_{C(I;H^{s-1})}
\Big\}
\end{multline*}
for all $t\in I$. Then, applying the Gronwall lemma yields
\begin{equation}
\begin{array}{ll}\label{51}
\ds \|\varphi  - \varphi_{n} \|^2_{C(I;H^{s})} +  \|(\varphi  - \varphi_{n})_t \|^2_{C(I;H^{s-1})} 
\\
\ds \le C_3 \Big\{ \eps^2 + \| \varphi^{(1)}- \varphi_{n}^{(1)}\|^2_{H^{s-1}} 
 + \| \varphi^{(0)} - \varphi_{n}^{(0)} \|^2_{H^{s}}    + M(\eps)\|\varphi_n -\varphi\|^2_{C(I;H^{s-1})}
\Big\}.
\end{array}
\end{equation}
Given any $\eps'>0$, let $\eps=\eps(\eps')$ be such that $C_3\eps^2<\eps'/3$. With this fixed $\eps$ in $M(\eps)$, and taking account of Proposition \ref{convs-1}, let $n_0$ be such that, for any $n\ge n_0$,
\[
C_3\left\{\| \varphi^{(1)}- \varphi_{n}^{(1)}\|^2_{H^{s-1}} 
 + \| \varphi^{(0)} - \varphi_{n}^{(0)} \|^2_{H^{s}}  \right\} <\eps'/3,
\]
\[
C_3 M(\eps)\|\varphi_n -\varphi\|^2_{C(I;H^{s-1})}<\eps'/3.
\]
It follows from \eqref{51} that
\[
\ds  \|\varphi  - \varphi_{n} \|^2_{C(I;H^{s})} +  \|(\varphi  - \varphi_{n})_t \|^2_{C(I;H^{s-1})} 
 <\eps' \qquad \forall n\ge n_0.
\]
This concludes the proof of Theorem \ref{main}.
\end{proof}


\appendix                                                                                                                                                                   
\section{Some commutator estimates}\label{stima_commutatore}                                                                                                    

\begin{lemma}\label{lemma_comm}
 For $\t>1/2$ there exists a constant $C_\t>0$ such that
 \begin{eqnarray}
\Vert \left[\mathbb H\,;\,v\right]f\Vert_{L^2(\mathbb T)}\le C_\t\Vert v\Vert_{H^\t(\mathbb T)}\Vert f\Vert_{L^2(\mathbb T)}\,,\quad\forall\,v\in H^\t(\mathbb T)\,,\,\,\forall\,f\in L^2(\mathbb T)\,,
\label{stima_comm_1}\end{eqnarray}
where $\left[\mathbb H\,;\,v\right]$ is the commutator between the Hilbert transform $\mathbb H$ and the multiplication by $v$.
\end{lemma}
\begin{proof}
The proof can be found in \cite{M-S-T:ONDE1}.
\end{proof}
\begin{lemma}\label{lemma_comm_ale_2}
For every real $\t\ge 0$ and integer $p\ge 0$ there exists a constant $C_{\t,p}>0$ such that
 for all functions $v\in H^{\t+p}(\mathbb T)$ and $f\in H^1(\mathbb T)$
\begin{equation}\label{stima_comm_p}
\Vert \left[\mathbb H\,;\,v\right]\partial^p_x f\Vert_{H^\t(\mathbb T)}\le C_{\t,p}\Vert \partial^{p}_x v\Vert_{H^\t(\mathbb T)}\Vert f\Vert_{H^1(\mathbb T)}\,.
\end{equation}
\end{lemma}
\begin{proof}
The proof can be found in \cite{M-S-T:ONDE2}.
\end{proof}
\begin{lemma}\label{commutatore_ds}
For every real $\t\ge 1$ and $\sigma>1/2$ there exists a constant $C_{\t,\sigma}>0$ such that
\begin{itemize}
\item[i.] for all $f\in H^{\t-1}(\mathbb T)\cap H^\sigma(\mathbb T)$ and $v\in H^\t(\mathbb T)\cap H^2(\mathbb T)$
\begin{equation}\label{stima_comm_ds}
\Vert \left[\langle\partial_x\rangle^\t\,;\,v\right]f\Vert_{L^2(\mathbb T)}\le C_{\t,\sigma}\left\{\Vert v\Vert_{H^\t(\mathbb T)}\Vert f\Vert_{H^\sigma(\mathbb T)}+\Vert v_x\Vert_{H^1(\mathbb T)}\Vert f\Vert_{H^{\t-1}(\mathbb T)}\right\}\,;
\end{equation}
\item[ii.] for all $f\in H^{\t-1}(\mathbb T)$ and $v\in H^{\t+\sigma}(\mathbb T)\cap H^2(\mathbb T)$
\begin{equation}\label{stima_comm_ds1}
\Vert \left[\langle\partial_x\rangle^\t\,;\,v\right]f\Vert_{L^2(\mathbb T)}\le C_{\t,\sigma}\left\{\Vert v\Vert_{H^{\t+\sigma}(\mathbb T)}\Vert f\Vert_{L^2(\mathbb T)}+\Vert v_x\Vert_{H^1(\mathbb T)}\Vert f\Vert_{H^{\t-1}(\mathbb T)}\right\}\,.
\end{equation}
\end{itemize}
For all $\t\ge 1$ there exists a positive constant $C_\t$ such that for all $f\in H^{\t-1}(\mathbb T)\cap H^{1/2}(\mathbb T)$ and $v\in H^{\t+1/2}(\mathbb T)\cap H^2(\mathbb T)$
\begin{equation}\label{stima_comm_ds2}
\Vert \left[\langle\partial_x\rangle^\t\,;\,v\right]f\Vert_{L^2(\mathbb T)}\le C_{\t}\left\{\Vert v\Vert_{H^{\t+1/2}(\mathbb T)}\Vert f\Vert_{H^{1/2}(\mathbb T)}+\Vert v_x\Vert_{H^1(\mathbb T)}\Vert f\Vert_{H^{\t-1}(\mathbb T)}\right\}\,.
\end{equation}
\end{lemma}
\begin{proof}
For all $k\in\mathbb Z$ we compute
\begin{equation}\label{coeff_fourier_comm}
\begin{split}
\widehat{\left[\langle\partial_x\rangle^\t\,;\,v\right]f}(k)&=\langle k\rangle^\t\widehat{vf}(k)-\widehat{v\langle\partial_x\rangle^\t f}(k)\\
&=\frac1{2\pi}\langle k\rangle^\t\sum\limits_{\ell}\widehat{v}(k-\ell)\widehat f(\ell)-\frac1{2\pi}\sum\limits_{\ell}\widehat{v}(k-\ell)\langle\ell\rangle^\t\widehat{f}(\ell)\\
&=\frac1{2\pi}\sum\limits_{\ell}\left(\langle k\rangle^\t-\langle\ell\rangle^\t\right)\widehat{v}(k-\ell)\widehat{f}(\ell)\,.
\end{split}
\end{equation}
On the other hand we have
\begin{equation*}
\langle k\rangle^\t-\langle\ell\rangle^\t=\int_0^1\frac{d}{d\theta}\left(\langle \ell+\theta(k-\ell)\rangle^\t\right)\,d\theta
=(k-\ell)\int_0^1 D\left(\langle \cdot\rangle^\t\right)(\ell+\theta(k-\ell))d\theta\,,
\end{equation*}
where $D$ denotes the derivative of the function $\langle \cdot\rangle^\t$. Combining the preceding with the estimate
\begin{equation}\label{derivate_ds}
\left\vert\frac{d}{d\xi}\langle \xi\rangle^\t\right\vert\le C_\t\langle\xi\rangle^{\t-1}\,,\quad\forall\,\xi\in\mathbb R\,,
\end{equation}
then gives
\begin{equation}\label{stima_diff_ds}
\vert\langle k\rangle^\t-\langle\ell\rangle^\t\vert\le\vert k-\ell\vert\int_0^1 \vert D\left(\langle \cdot\rangle^\t\right)(\ell+\theta(k-\ell))\vert d\theta\le C_\t\vert k-\ell\vert\int_0^1\langle \ell+\theta(k-\ell)\rangle^{\t-1}\,d\theta\,.
\end{equation}
Using \eqref{stima_diff_ds}, from \eqref{coeff_fourier_comm} we get
\begin{equation}\label{stima_coeff_comm}
\begin{split}
\vert\widehat{\left[\langle\partial_x\rangle^\t\,;\,v\right]f}(k)\vert&\le\frac1{2\pi}\sum\limits_{\ell}\left\vert\langle k\rangle^\t-\langle\ell\rangle^\t\right\vert\vert\widehat{v}(k-\ell)\vert\vert\widehat{f}(\ell)\vert\\
&\le C_\t\sum\limits_{\ell}\int_0^1\vert k-\ell\vert\langle \ell+\theta(k-\ell)\rangle^{\t-1}\vert\widehat{v}(k-\ell)\vert\vert\widehat{f}(\ell)\vert\,d\theta\,.
\end{split}
\end{equation}
Since the function $\langle\zeta\rangle^{\t-1}$ is sub-additive and $0\le\theta\le 1$, we have
\begin{equation}\label{sub-add}
\langle\ell+\theta(k-\ell)\rangle^{\t-1}\le C_\t\left\{\langle\theta(k-\ell)\rangle^{\t-1}+\langle\ell\rangle^{\t-1}\right\}\le C_\t\left\{\langle k-\ell\rangle^{\t-1}+\langle\ell\rangle^{\t-1}\right\}\,,
\end{equation}
with positive constant $C_\t$ depending only on $\t$. Using \eqref{sub-add} to estimate the right-hand side of \eqref{stima_coeff_comm} then gives
\begin{equation}\label{stima_coeff_comm1}
\begin{split}
\vert\widehat{\left[\langle\partial_x\rangle^\t\,;\,v\right]f}(k)\vert&
\le C_\t\sum\limits_{\ell}\left\{\int_0^1\langle k-\ell\rangle^{\t}\vert\widehat{v}(k-\ell)\vert\vert\widehat{f}(\ell)\vert\,d\theta+\int_0^1\vert k-\ell\vert \vert\widehat{v}(k-\ell)\langle\ell\rangle^{\t-1}\vert\vert\widehat{f}(\ell)\vert\,d\theta\right\}\\
&\le C^\prime_\t\left\{\left(\vert\widehat{\langle\partial_x\rangle^\t v}\vert\ast\vert\widehat{f}\vert\right) (k)+\left(\vert\widehat{v_x}\vert\ast\vert\widehat{\langle\partial_x\rangle^{\t-1}f}\vert\right)(k)\right\}\,.
\end{split}
\end{equation}
Using Parseval's identity and Young's inequality with $\left\{\vert\widehat{\langle\partial_x\rangle^\t v}(k)\vert\right\}\in\ell^2$, $\left\{\vert\widehat{f}(k)\vert\right\}\in\ell^1$, $\left\{\vert\widehat{v_x}(k)\vert\right\}\in\ell^1$,  $\left\{\vert\widehat{\langle\partial_x\rangle^{\t-1} f}(k)\vert\right\}\in\ell^2$, from \eqref{stima_coeff_comm1} we derive
\begin{equation}\label{stima_coeff_comm2}
\begin{split}
\Vert\left[\langle\partial_x\rangle^\t\,;\,v\right]f\Vert_{L^2(\mathbb T)}\le C^\prime_\t\left\{\Vert\{\vert\widehat{\langle\partial_x\rangle^\t v}\vert\}\Vert_{\ell^2}\Vert\{\vert\widehat{f}\vert\}\Vert_{\ell^1}+\Vert\{\vert\widehat{v_x}\vert\}\Vert_{\ell^1}\Vert\{\vert\widehat{\langle\partial_x\rangle^{\t-1}f}\vert\}\Vert_{\ell^2}\right\}\,.
\end{split}
\end{equation}
We get the first inequality \eqref{stima_comm_ds} of Lemma \ref{commutatore_ds}, by using once again Parseval's identity and the estimates
\begin{equation}\label{stima_coeff_comm3}
\Vert\{\vert\widehat{f}\vert\}\Vert_{\ell^1}\le C_\sigma\Vert f\Vert_{H^\sigma(\mathbb T)}\,,\qquad \Vert\{\vert\widehat{v_x}\vert\}\Vert_{\ell^1}\le C\Vert v_x\Vert_{H^1(\mathbb T)}\, .
\end{equation}

To get the second inequality \eqref{stima_comm_ds1} it is sufficient to interchange the role of the sequences $\{\vert\widehat{\langle\partial_x\rangle^\t v}\vert\}$ and $\{\vert\widehat{f}\vert\}$ when we apply Young's inequality to the first term in the right-hand side of \eqref{stima_coeff_comm1}, by taking the $\ell^1$-norm of  $\{\vert\widehat{\langle\partial_x\rangle^\t v}\vert\}$ and the $\ell^2$-norm of $\{\vert\widehat{f}\vert\}$; then the $\ell^1-$norm of $\{\vert\widehat{\langle\partial_x\rangle^\t v}\vert\}$ is estimated again by 
the first inequality in \eqref{stima_coeff_comm3}
\begin{equation*}
\Vert\{\vert\widehat{\langle\partial_x\rangle^\t v}\vert\}\Vert_{\ell^1}\le C_\sigma\Vert \langle\partial_x\rangle^\t v\Vert_{H^\sigma(\mathbb T)}\le C_\sigma\Vert v\Vert_{H^{\t+\sigma}(\mathbb T)}\,.
\end{equation*}

To obtain the last inequality \eqref{stima_comm_ds2}, the $\ell^2-$norm of the sequence $\{\vert\widehat{\langle\partial_x\rangle^\t v}\vert\ast\vert\widehat{f}\vert\}$ in the right-hand side of \eqref{stima_coeff_comm1} is estimated by Young's inequality as
\begin{equation}\label{stima_coeff_comm4}
\Vert \{\vert\widehat{\langle\partial_x\rangle^\t v}\vert\ast\vert\widehat{f}\vert\} \Vert_{\ell^2}\le \Vert \{\vert\widehat{\langle\partial_x\rangle^\t v}\vert\}\Vert_{\ell^{4/3}}\Vert\{\vert\widehat{f}\vert\}\Vert_{\ell^{4/3}}\,;
\end{equation}
then recalling that for every $p\in]1,2]$ a positive constant $C_p$ exists such that
\begin{equation}\label{stima_coeff_comm5}
\Vert \{\widehat{f}\}\Vert_{\ell^p}\le C_p\Vert f\Vert_{H^{1/2}(\mathbb T)}\,,\quad\forall\,f\in H^{1/2}(\mathbb T)\,,
\end{equation}
the $\ell^{4/3}-$norms in the right-hand side of \eqref{stima_coeff_comm4} are estimated as
\begin{equation}\label{stima_coeff_comm6}
\Vert \{\vert\widehat{\langle\partial_x\rangle^\t v}\vert\}\Vert_{\ell^{4/3}}\le C\Vert \langle\partial_x\rangle^\t v\Vert_{H^{1/2}(\mathbb T)}\le C\Vert v\Vert_{H^{\t+1/2}(\mathbb T)}\,,\qquad \Vert \{\vert\widehat{f}\vert\}\Vert_{\ell^{4/3}}\le C\Vert f\Vert_{H^{1/2}(\mathbb T)}
\end{equation}
(that is \eqref{stima_coeff_comm5} with $p=\frac{4}{3}$). Then \eqref{stima_comm_ds2} follows from gathering the estimates \eqref{stima_coeff_comm4}, \eqref{stima_coeff_comm6} and repeating for the rest the same calculations as above.
\end{proof}




\end{document}